\documentclass[a4paper,12pt]{amsart}

\usepackage[sc]{mathpazo}
\usepackage[T1]{fontenc}

\usepackage{graphicx}
\usepackage{amsmath}
\usepackage{amssymb}
\usepackage[active]{srcltx}
\usepackage{hyperref}
\usepackage{bbm}
\usepackage{enumerate}
\usepackage{mathrsfs}

\newtheorem{prop}{Proposition}%

\theoremstyle{definition}

\theoremstyle{remark}
\theoremstyle{plain}

\def\EE{{\mathbb E}}

\def\NN{{\mathbb N}}

\def\PP{{\mathbb P}}
\def\RR{{\mathbb R}}

\def\ZZ{{\mathbb Z}}

\def\curA{{\mathscr A}}

\def\scrM{{\mathcal M}}
\def\scrN{{\mathcal N}}

\def\L{\operatorname{L{}}}
\def\Leb{\operatorname{Leb}}

\def\Prob{\operatorname{Prob}}

\newcommand{\ind}[1]{\ensuremath{{\mathbbm{1}}{\big(#1\big)}}}

\def \toas {\,\,\buildrel\text{\rm a.s.}\over\longrightarrow\,\,}
\def \todist {\,\,\buildrel\text{\rm d}\over\longrightarrow\,\,}
\def \toweak {\,\,\buildrel\text{\rm w}\over\longrightarrow\,\,}
\def \eqdist {\buildrel\text{\rm d}\over =}

\def\time{R}

\title{Entry and return times for semi-flows}
\author{Jens Marklof}
\address{Jens Marklof, School of Mathematics, University of Bristol, Bristol BS8 1TW, U.K.\newline \rule[0ex]{0ex}{0ex} \hspace{8pt}{\tt j.marklof@bristol.ac.uk}}
\date{20 May 2016/25 October 2016. To appear in Nonlinearity.}
\thanks{The research leading to these results has received funding from the European Research Council under the European Union's Seventh Framework Programme (FP/2007-2013) / ERC Grant Agreement n. 291147.}
\subjclass[2010]{37A10, 60G55}

\begin{document}

\begin{abstract}
Haydn, Lacroix and Vaienti [Ann. Probab. 33 (2005)] proved that, for a given ergodic map, the entry time distribution converges in the small target limit, if and only if the corresponding return time distribution converges. The present note explains how entry and return times can be interpreted in terms of stationary point processes and their Palm distribution. This permits a generalization of the results by Haydn et al.\ to non-ergodic maps and continuous-time dynamical systems.
\end{abstract}

\maketitle

\section{Introduction}

There is currently much interest in understanding the entry and return time distributions of various classes of dynamical systems, in particular in the case of small target sets; cf.~\cite{Abadi11,Chazottes13,Freitas10,Freitas13,Freitas14,Freitas16,Haydn13,Haydn14a,Haydn14b,Lucarini16,Rousseau14} and references therein. These statistics may be viewed as quantitative refinements of the classical Poincar\'e recurrence, as pointed out by Kac in his study of return times for ergodic maps and discrete stochastic processes \cite{Kac47}. We will here review the relationship between the entry and return time distributions for dynamical systems, 
and relate them to known facts for stationary point processes and their Palm distribution. This yields in particular an extension of the results by Haydn et al.\ \cite{Haydn05} on the convergence of entry and return time distributions, as well as the processes generated by them \cite{CK,Zweimueller16}, to continuous-time dynamical systems with no assumption on ergodicity. 

The plan for this paper is as follows. Rather than dealing with discrete-time systems, we consider their suspension semi-flows with general integrable roof functions (Section \ref{sec:suspensions}). This gives a natural interpretation of the entry and return times to a given set in terms of a random point process on the positive real line (Section \ref{sec:point}). Invariance of the measure under the semi-flow  implies stationarity of the process, which can be naturally extended to a stationary point process on the full real line (Section \ref{sec:stationarity}).  The return times are obtained as the entry times corresponding to the Palm distribution of the original point process. The formula found by Haydn et al.\ \cite{Haydn05} follows from a fundametal identity in Palm distribution theory (Section \ref{sec:palm}). This identity is a close relative of Santalo's formula for the free path length in billiards and other geometric settings \cite{Chernov97,Santalo04,Stoyanov16}, and yields a transparent proof of Kac's formula (Eq.~\eqref{eq:KacX} in Section \ref{sec:palm}) for the mean return time \cite{Kac47,Meiss97,Wright61}. It also allows us to prove the equivalence of the convergence of entry and return time probabilities for sequences of target sets and dynamical systems, thus extending \cite{Haydn05} (see also \cite{Abadi11,CK,Haydn07,Haydn14b,Zweimueller16}) to non-ergodic maps and semi-flows  (Section~\ref{sec:shrinking}). The general equivalence of convergence-in-distribution of the point processes generated by the entry times and return times, respectively, follows from convergence results for Palm distributions in \cite{Kallenberg73}; see Section \ref{sec:shrinking} for more details. 

The connection between Ambrose's representation of flows as suspensions \cite{Ambrose41,Ambrose42} and the Palm distributions of stationary point processes was first pointed out by Papangelou \cite{Papangelou70} and Hanen \cite{Hanen71}; see also \cite{Benveniste75,deSam75,Geman75}.
The relationship between entry and return time distributions and their convergence has been exploited in numerous applications. These include the limit distributions for gap statistics of number-theoretic sequences \cite{Athreya16,Athreya14,Elkies04,Marklof07,Marklof13,gaps,vis,hyper}, the role of Palm distributions in the statistics of quantum energy levels \cite{Minami07},
spectra of quantum graphs \cite{Barra00}, gap times in unimolecular reaction rates \cite{Ezra09}, and the Boltzmann-Grad limit of the periodic Lorentz gas \cite{Boca07,Caglioti10,partI}. 

{\em Acknowledgements:} I thank Carl Dettmann, Omri Sarig, Andreas Str\"ombergsson, Mike Todd, Corinna Ulcigrai, Steve Wiggins and Roland Zweim\"uller for their comments on the first drafts of this paper.

\section{Suspensions}\label{sec:suspensions}

Following \cite{Ambrose41,Ambrose42,Krengel69} we may represent any given measurable semi-flow $\varphi^t$ as the suspension over a suitable map $T$ as follows. 
Let $(Y,\curA,\mu)$ be a probability space and $T:Y\to Y$ a measurable map preserving $\mu$. Fix a function $r\in\L^1(Y,\mu)$ such that $\inf r>0$. Let $X=\{ (y,s) \mid y\in Y,\; 0\leq s < r(y)\}$, and set
\begin{equation}
S_0(y)=0,\qquad S_n(y) =\sum_{j=0}^{n-1} r(T^j y) \quad (n\geq 1).
\end{equation}
The {\em suspension} (also called {\em special flow} or {\em flow under a function}) of $T$ with {\em roof function} $r$ is the semi-flow defined for $t\geq 0$ by
\begin{equation}
\varphi^t : X\to X, \qquad (y,s)\mapsto (T^n y,s+t-S_n(y)) 
\end{equation}
if $S_n(y)\leq s+t <S_{n+1}(y)$, $n=0,1,2,\ldots$.
The corresponding $\varphi^t$-invariant probability measure $\nu$ on $X$ is given by
\begin{equation}\label{Santalo}
\int_X f(x) \nu(dx) = \frac{1}{\overline r}\int_Y \bigg( \int_0^{r(y)} f((y,s)) ds \bigg) \mu(dy)
\end{equation}
for every measurable $f:X\to\RR_{\geq 0}$, where
\begin{equation}
\overline r:= \mu r :=\int_Y r(y)\mu(dy).
\end{equation} 
In various geometric contexts (such as billiard dynamics or geodesic flows), identities analogous to \eqref{Santalo} are often referred to as Santalo's formula \cite{Chernov97,Santalo04,Stoyanov16}.

Given $x\in X$, $D\in\curA$, we define the sequence of entry/return times $0< t_1<t_2<\ldots$ as the elements of the (possibly empty) set
\begin{equation}
H(x,D) = \{ t> 0 \mid \varphi^t x \in D \times \{0\} \} .
\end{equation}
We call $t_j=t_j(x)$ the $j$th entry time if $x\notin D \times \{0\}$, and the $j$th return time if $x\in D \times \{0\}$.
Note that $t_{j+1}-t_j\geq \inf r$ for all $j$. We first discuss the entry and return time distribution for fixed $D$, then (in Section \ref{sec:shrinking}) for a sequence of target sets $D_n$ and, more generally, sequences of suspension semi-flows.

\section{Point processes}\label{sec:point}

Let $\scrM(\RR_{>0})$ (resp.\ $\scrM(\RR)$) be the space of locally finite Borel measures on $\RR_{>0}$ (resp.\ $\RR$), equipped with the vague topology. We denote by $\scrN(\RR_{>0})$ (resp.\ $\scrN(\RR)$) the closed subset of integer-valued measures, i.e., the set of $\zeta$ such that $\zeta B\in\ZZ\cup\{\infty\}$ for any Borel set $B$.
A {\em point process} on $\RR_{>0}$ (resp. $\RR$) is a random measure in $\scrN(\RR_{>0})$ (resp.\ $\scrN(\RR)$). Given an integer-valued measure $\zeta$, we can write
\begin{equation}
\zeta=\sum_j \delta_{\tau_j(\zeta)}
\end{equation}
where the $\tau_j$ are real-valued functions of $\zeta$. We use the convention that  $\tau_j\leq \tau_{j+1}$, and $\tau_0\leq 0 < \tau_1$ if there are non-positive $\tau_j$. A point process is called simple if $\sup_{t} \zeta\{t\}\leq 1$ a.s.

The push-forward $P_*\nu$ of the invariant measure $\nu$ under the measurable map
\begin{equation}\label{Px}
P: X \to \scrM(\RR_{>0}), \qquad x\mapsto  \xi=\sum_{t\in H(x,D)} \delta_t  = \sum_j \delta_{t_j(x)},
\end{equation}
defines a simple point process $\xi$. This process depends on the choice of $D\in\curA$. The $j$th entry time of the suspension is precisely given by $t_j(x)=\tau_j(\xi)$, where $x$ is distributed according to $\nu$. We will see in Proposition \ref{prop:see} that the point process $P(y,0)$ given by the return times $t_j(y,0)$, with $y\in D$ distributed according to the probability measure $\mu(D)^{-1} \mu$, has the same distribution as a point process $\eta$ given by the Palm distribution of $\xi$.

\section{Stationarity} \label{sec:stationarity}

For $u\in\RR_{\geq 0}$, define the shift operator $\theta^u$ on $\scrM(\RR_{>0})$ by $\theta^u\zeta B=\zeta(B+u)$ for every Borel set $B\subset\RR_{>0}$. If $\zeta\in\scrN(\RR_{>0})$, we have
\begin{equation}
\zeta= \sum_{j=1}^\infty \delta_{\tau_j(\zeta)}, \qquad \theta^u\zeta= \sum_{j=j_0}^\infty \delta_{\tau_j(\zeta)-u},
\end{equation}
where $j_0 = \min\{ j \mid \tau_j-u >0 \} $. We say a random measure $\zeta\in\scrM(\RR_{>0})$ is stationary if $\theta^u\zeta\eqdist \zeta$ for all $u\in\RR_{\geq 0}$. 
Similarly, for $u\in\RR$, define the shift operator $\theta^u$ on $\scrM(\RR)$ by $\theta^u\zeta B=\zeta(B+u)$ for every Borel set $B\subset\RR$. We say a random measure $\zeta$ in $\scrM(\RR)$ is stationary if $\theta^u\zeta\eqdist \zeta$ for all $u\in\RR$. 
If $\zeta_0$ in $\scrM(\RR_{>0})$ is a stationary random measure, then (with $\zeta_0$ now viewed as an element in $\scrM(\RR)$) the limit $\theta^u\zeta_0 \todist \zeta$ ($u\to\infty$) exists  in $\scrM(\RR)$ and defines a stationary random measure, whose restriction to $\scrM(\RR_{>0})$ has the same distribution as $\zeta_0$. We call $\zeta$ the {\em stationary extension} of $\zeta_0$ to $\RR$. 

The intensity measure of a stationary random measure $\zeta$ is $I_\zeta \Leb$, where the intensity is given by
\begin{equation}
I_\zeta = \frac{\EE \zeta(0,\time]}{\time},
\end{equation}
which, by stationarity, is independent of the choice of $\time>0$. 

\begin{prop}\label{prop:stat}
The simple point process $\xi=P(x)$, with $x\in X$ distributed according to $\nu$, is stationary on $\RR_{>0}$ with intensity $I_\xi=\overline r^{-1} \mu(D)$. 
\end{prop}

\begin{proof}
Since $\nu$ is $\varphi^u$-invariant, we have for any $u\geq 0$
\begin{equation}
\xi= \sum_{t\in H(x,D)} \delta_t\eqdist \sum_{t\in H(\varphi^u x,D)} \delta_t 
= \theta^u\xi,
\end{equation}
where the last relation follows from
\begin{align}
H(\varphi^u x,D) & = \{ t> 0 \mid \varphi^{t+u} x \in D \times \{0\} \} \\
& = \{ t> u \mid \varphi^t x \in D \times \{0\} \} - u \\
& =(H( x,D) - u)\cap\RR_{>0} .
\end{align}
This shows that $\xi$ is stationary on $\RR_{>0}$. Its intensity is 
\begin{equation}
I_\xi = \EE  \frac{\#\{ j \mid 0< \tau_j(\xi) \leq \time \}}{\time} =\int_X \frac{\#\{ j \mid 0< t_j(x) \leq \time \}}{\time} \nu(dx). 
\end{equation}
Now, for any $\epsilon\in(0,\inf r)$,
\begin{equation}\label{fourteen}
\#\{ j \mid 0< t_j(x) \leq \time \} = \frac1\epsilon \int_0^\time \ind{\varphi^t(x)\in D\times [0,\epsilon]} dt + E_\time(x)
\end{equation}
where $|E_\time(x)|\leq 1$. Since $\varphi^t$ preserves $\nu$, we have
\begin{align}
\int_X \ind{\varphi^t(x)\in D\times [0,\epsilon]} \nu(dx) & = \nu(\varphi^{-t}(D\times [0,\epsilon])) \\
& = \nu(D\times [0,\epsilon]) 
= \epsilon\,  \overline r^{-1} \mu(D).
\end{align}
We conclude $I_\xi=\overline r^{-1} \mu(D) +O(\time^{-1})$ for any $\time>0$.
Taking $\time\to\infty$ proves that $I_\xi=\overline r^{-1} \mu(D)$.
\end{proof}

Note that a random measure and its stationary extension have the same intensity, so in particular if $\xi$ is the stationary extension of $P(x)$ to $\RR$, then $I_\xi=\overline r^{-1} \mu(D)$.
Let
\begin{equation}\label{Z1}
Z_D^1=\{ y\in Y \mid T^j y \in D \text{ for some $j\in \NN$}\} = \bigcup_{j=1}^\infty T^{-j} D
\end{equation}
be the set of points in $Y$ whose forward orbit intersects $D$. 

\begin{prop}\label{prop:as}
Let $\xi$ be the stationary extension of $P(x)$ to $\RR$, with $x\in X$ distributed according to $\nu$. Then
\begin{equation}\label{Z11}
\PP(\xi\neq 0) = \mu(Z_D^1).
\end{equation}
\end{prop}

\begin{proof}
We have $\xi\neq 0$ if and only if $x\in\{ (y,s) \mid y\in Z_D^1,\, 0\leq s <r(y)\}$, and the claim follows.
\end{proof}

Let
\begin{equation}
Z_D^\infty=\{ y\in Y \mid T^j y \in D \text{ for infinitely many  $j\in \NN$}\} = \bigcap_{k=1}^\infty \bigcup_{j=k}^\infty T^{-j} D
\end{equation}
be the set of points in $Y$ whose forward orbit intersects $D$ infintely often. Now $\PP(\xi[0,\infty)=\infty) = \mu(Z_D^\infty)$ if and only if $x\in\{ (y,s) \mid y\in Z_D^\infty,\, 0\leq s <r(y)\}$.
The zero-infinity law for stationary random measures on $\RR$ \cite[Lemma 11.1]{Kallenberg02} gives
\begin{equation}
\PP(\xi\neq 0) = \PP(\xi[0,\infty)=\infty) = \mu(Z_D^\infty).
\end{equation}
By Proposition \ref{prop:as}, this shows that $\mu(Z_D^1)=\mu(Z_D^\infty)$, which is a formulation of the Poincar\'e recurrence theorem. 
Note that $Z_D^\infty$ is $T$-invariant. Thus, if $T$ is ergodic, then $\mu(Z_D^\infty)\in\{0,1\}$ and $\mu(Z_D^\infty)=1$ if and only if $\mu(D)>0$.

An example of a non-ergodic map, to which the above proposition applies, is 
\begin{equation}
T: \RR^2/\ZZ^2 \to\RR^2/\ZZ^2 , \qquad (x_1,x_2)\mapsto (x_1+x_2,x_2) ,
\end{equation}
where $\mu$ is the Haar probability measure of $\RR^2/\ZZ^2$. $T$ is evidently not ergodic; for instance $D=\RR/\ZZ \times D_2$ is $T$-invariant for any measurable $D_2\subset\RR/\ZZ$ and $\mu(D)=\Leb D_2$ can take any value in $[0,1]$. On the other hand $\mu(Z_D^1)=1$, if for instance $D=D_1\times \RR/\ZZ$, for any measurable $D_1\subset\RR/\ZZ$ with $\Leb D_1>0$. The entry and return times of this map (and its higher-dimensional analogues) play a central role in the periodic Lorentz gas \cite{Boca07,Caglioti10,partI}.

\section{Palm distributions}\label{sec:palm}

The Palm distribution $Q_\xi$ of a stationary random measure $\xi$ on $\RR$ with finite and positive intensity is defined by the relation \cite[Chapter 11]{Kallenberg02}
\begin{equation}\label{eq:Q}
Q_\xi f = \frac{1}{I_\xi\Leb B} \EE \int_B f(\theta^u\xi) \xi(du)  ,
\end{equation}
for any measurable $f: \scrM(\RR)\to\RR_{\geq 0}$, where $B\subset\RR$ is a given Borel set with $\Leb B>0$. It follows from the stationarity of $\xi$ that this definition is in fact independent of the choice of $B$. 
Let $\eta$ be a random measure distributed according to the Palm distribution $Q_\xi$, i.e., $\EE f(\eta)=Q_\xi f$.
If $\xi$ is a point process, we can write $\xi=\sum_j \delta_{\tau_j(\xi)}$, and so
\begin{equation}\label{eq:QQQQQ}
\EE f(\eta) = \frac{1}{I_\xi\Leb B} \EE  \sum_j  \ind{\tau_j(\xi)\in B} f\bigg(\sum_i \delta_{\tau_i(\xi)-\tau_j(\xi)} \bigg) .
\end{equation}
This shows that $\eta$ is a point process and furthermore that, if $\xi$ is a simple point process,  then $\eta$ is a simple point process, too, and $\eta\{0\}=1$ a.s. The stationarity of $\xi$ implies that $\eta$ is {\em cycle-stationary}, i.e.\ $\eta$ has the same distribution as the point process 
$\theta^{\tau_j(\eta)} \eta = \sum_i \delta_{\tau_i(\eta)-\tau_j(\eta)}$ for any $j$.

A basic example is when $\xi$ is a Poisson process with constant intensity $I_\xi$. In this case $\eta\eqdist \delta_0+\xi$. This relation is in fact unique to the Poisson process (Slivnyak's theorem): If $\xi$ is a stationary process on $\RR$ and $\delta_0+\xi$ is distributed according to $Q_\xi$, then $\xi$ is a homogeneous Poisson process.

If $\xi$ is the stationary extension of $P(x)$ to $\RR$ with $P$ as in \eqref{Px}, then \eqref{eq:QQQQQ} shows that any point process $\eta$ with distribution $Q_\xi$ is supported on the return times. We in fact have the following relation with the return time process $P(y,0)$, where $y$ is uniformly distributed on $D$.

\begin{prop}\label{prop:see}
Let $\xi$ be the stationary extension of $P(x)$ to $\RR$, with $x\in X$ distributed according to $\nu$, and $\eta$ a point process distributed according to the Palm distribution $Q_\xi$. Then, for every measurable $f:\scrN(\RR)\to\RR_{\geq 0}$ supported on $\scrN(\RR_{>0})$,
\begin{equation}\label{eq:QQQQQTTTT}
\EE f(\eta)  =  \frac{1}{\mu(D)}  \int_D f(P(y,0))  \mu(dy) .
\end{equation}
\end{prop}

\begin{proof}
For any $\epsilon\in(0,\inf r]$, we have $\xi (-\epsilon,0]\leq 1$ and so \eqref{eq:QQQQQ} simplifies to
\begin{equation}\label{eq:QQQQQ5}
\EE f(\eta) = \frac{1}{I_\xi\epsilon} \EE \ind{-\epsilon<\tau_0(\xi)\leq 0} f\bigg(\sum_i \delta_{\tau_i(\xi)-\tau_0(\xi)} \bigg) .
\end{equation}
In geometric terms, $-\epsilon<\tau_0(\xi)\leq 0$ restricts the expectation to orbits that have just hit the section $D$ within the time interval $[0,\epsilon)$. Since, by assumption, $f$ is supported on $\scrN(\RR_{>0})$,
\eqref{eq:QQQQQ5} becomes
\begin{equation}\label{eq:QQQQQTTTT0}
\EE f(\eta) 
 = \frac{1}{\overline r I_\xi\epsilon} \int_D \bigg( \int_0^{r(y)} \ind{0\leq -t_0(y,s) <\epsilon} f\bigg(\sum_i \delta_{t_i(y,s)-t_0(y,s)} \bigg)  ds \bigg)  \mu(dy).
\end{equation}
For $s<\epsilon\leq\inf r$, we have $t_j(y,s)+s=t_j(y,0)$ (in particular $t_0(y,s)=-s$), and so
\begin{equation}
\sum_i \delta_{t_i(y,s)-t_0(y,s)} = \sum_i \delta_{t_i(y,0)}  = P(y,0).
\end{equation}
Hence
\begin{align}\label{eq:QQQQQTTTT00TT}
\EE f(\eta) 
& = \frac{1}{\overline r I_\xi\epsilon} \int_D \bigg( \int_0^{r(y)} \ind{0\leq s<\epsilon} f(P(y,0))  ds \bigg)  \mu(dy) \\
& =\frac{1}{\overline r I_\xi} \int_D  f(P(y,0))  \mu(dy),
\end{align}
and \eqref{eq:QQQQQTTTT}
follows from $\overline r I_\xi=\mu(D)$; recall Proposition \ref{prop:stat}. 
\end{proof}


The following inversion formula is a fundamental result of Palm distribution theory \cite[Proposition 11.3 (iii)]{Kallenberg02}. 

\begin{prop}\label{prop:one}
Assume $\xi$ is a simple, stationary point process on $\RR$ with positive finite intensity.
Let $\eta$ be a point process given by the Palm distribution $Q_\xi$. Then $\PP(\xi \in\,\cdot\,\mid \xi\neq 0)$ is uniquely determined by $\eta$, and, for any measurable $f: \scrN(\RR)\to\RR_{\geq 0}$,
\begin{equation}\label{eq:one}
\EE [f(\xi) \ind{\xi\neq 0}] = I_\xi\,  \EE \int_0^{\tau_1(\eta)}  f(\theta^u\eta) du .
\end{equation}
\end{prop}

We reproduce the elegant proof of Proposition \ref{prop:one} from \cite[p.\ 205]{Kallenberg02}.

\begin{proof}
In view of \eqref{eq:Q},
\begin{equation}\label{K1}
I_\xi \Leb(B)\, \EE f(\eta) = \EE \int_B f(\theta^u\xi) \xi(du)
\end{equation}
for any Borel set $B$ and measurable $f:\scrN(\RR)\to\RR_{\geq 0}$. By a monotone class argument (take finite linear combinations of functions of the form $g(\zeta,u)= f(\zeta)\chi_B(u)$ where $\chi_B$ is the indicator function of $B$), \eqref{K1} extends to 
\begin{equation}\label{K2}
I_\xi \EE \int_\RR g(\eta,u) du = \EE \int_\RR g(\theta^u\xi,u) \xi(du)
\end{equation}
for every measurable $g:\scrN(\RR)\times \RR\to\RR_{\geq 0}$. Using the function $g(\zeta,u)=h(\theta^{-u}\zeta,u)$ with measurable $h:\scrN(\RR)\to\RR_{\geq 0}$, and substituting $u$ by $-u$ in the first integral, yields
\begin{equation}\label{K3}
I_\xi \EE \int_\RR h(\theta^u \eta,-u) du = \EE \int_\RR h(\xi,u) \xi(du) .
\end{equation}
Apply this formula with
\begin{equation}
h(\zeta,u) = f(\zeta) \ind{\tau_0(\zeta)=u} ,
\end{equation}
and note that $\tau_0(\theta^u \eta)=-u$ if and only if $u\in[0,\tau_1(\eta))$. This proves \eqref{eq:one}. To show uniqueness of the conditional probability, note that in view of \eqref{eq:one} for any measurable set $A$, and $f$ its indicator function,
\begin{equation}
\Prob(\xi \in A \mid \xi\neq 0) = \frac{\EE [f(\xi) \ind{\xi\neq 0}]}{\EE \ind{\xi\neq 0}}
=\frac{ \EE \int_0^{\tau_1(\eta)}  f(\theta^u\eta) du}{\EE\tau_1(\eta)} .
\end{equation}
Here 
\begin{equation}\label{l8}
\EE \ind{\xi\neq 0}=I_\xi \EE\tau_1(\eta)
\end{equation}
follows from \eqref{eq:one} for $f\equiv1$.
\end{proof}

Note that \eqref{l8} can be written as
\begin{equation}\label{eq:Kac}
\EE \tau_1(\eta) = \frac{\PP(\xi\neq 0)}{I_\xi} = \overline r \, \frac{\mu(Z_D^1)}{\mu(D)}.
\end{equation}
This yields in view of  \eqref{Z1} and \eqref{eq:QQQQQTTTT} with $f(\eta)=\tau_1(\eta)$, the following identity for the mean return time for the semiflow $\varphi^t$:
\begin{equation}\label{eq:KacX}
 \int_D t_1(y,0)  \mu(dy)  = \overline r \, \mu\bigg(\bigcup_{j=1}^\infty T^{-j} D\bigg).
\end{equation}
Eq.~\eqref{eq:KacX} generalizes Kac's formula \cite{Kac47} to semi-flows, without any assumption on ergodicity. It is well known in billiard dynamics \cite{Chernov97} and other transport problems \cite{Ezra09,Rom-Kedar90}; see also also \cite{Varandas16}. As already pointed out by Wright \cite{Wright61} (cf.\ also \cite{Meiss97,Rom-Kedar90}), the ergodicity assumed by Kac is not required, if we restrict $Y$ to the set of points $Z_D^1$ whose forward trajectories intersect $D$. 
An elementary example of a non-ergodic $\xi$ and its Palm distribution is discussed in Appendix \ref{app:simple}.

Applying \eqref{eq:one} with the choice $f(\zeta)=\ind{\tau_1(\zeta)>\time}$ and \eqref{l8} yields 
\begin{align}
\PP(\tau_1(\xi)>\time \mid \xi\neq 0) 
& = \frac{1}{\EE\tau_1(\eta)} \EE \int_0^{\tau_1(\eta)} \ind{\tau_1(\eta) - u>\time} du \label{twentynine}\\
&=  \frac{1}{\EE\tau_1(\eta)} \EE \int_0^\infty \ind{\tau_1(\eta) >\time+u} du 
 , \label{HHH0}
\end{align}
and so
\begin{equation}\label{HHH} 
\PP(\tau_1(\xi)>\time \mid\xi\neq 0)  = \frac{1}{\EE\tau_1(\eta)}  \int_\time^\infty \PP(\tau_1(\eta) >u) du .
\end{equation}
Eq.~\eqref{twentynine} says that on $\xi\neq 0$ 
the random variable $\tau_1(\xi)$ has the same distribution as $\tau_1(\eta)-u$, where $u$ is uniformly distributed in the available gap $[0,\tau_1(\eta))$ with intensity $1/\EE\tau_1(\eta)$. 
Eq.\ \eqref{HHH} generalizes the formula found by Haydn et al.~\cite{Haydn05} for ergodic maps, and is well known in the literature in various other contexts, e.g.\ in the theory of renewal processes \cite[Chapter XI.4]{Feller}. Identities of this type are known as {\em Palm-Khinchin equations}. Further instances are obtained from \eqref{eq:one} by choosing the test function $f(\zeta)=\ind{\tau_j(\zeta)>\time}$, which yields
\begin{align}
\PP(\tau_j(\xi)>\time \mid \xi\neq 0) 
= & \frac{1}{\EE\tau_1(\eta)} \EE \int_0^{\tau_1(\eta)} \ind{\tau_j(\eta) - u>\time} du \label{twentynine111}\\
= & \frac{1}{\EE\tau_1(\eta)} \EE \int_0^{\tau_j(\eta)} \ind{\tau_j(\eta) - u>\time} du \label{firstI} \\
& - \frac{1}{\EE\tau_1(\eta)} \EE \int_{\tau_1(\eta)}^{\tau_j(\eta)} \ind{\tau_j(\eta) - u>\time} du \label{secondI}.
\end{align}
The first integral \eqref{firstI} is analogous to \eqref{HHH0}. For the second integral \eqref{secondI}, 
\begin{align}
\EE \int_{\tau_1(\eta)}^{\tau_j(\eta)} \ind{\tau_j(\eta) - u>\time} du
& =  \EE \int_{0}^{\tau_j(\eta)-\tau_1(\eta)} \ind{\tau_j(\eta) - \tau_1(\eta)- u>\time} du \\
& =  \EE \int_{0}^{\tau_{j-1}(\eta)} \ind{\tau_{j-1}(\eta) - u>\time} du,
\end{align}
by cycle-stationarity of the Palm distribution. In conclusion,
\begin{equation}
\PP(\tau_j(\xi)>\time \mid \xi\neq 0) 
= \frac{1}{\EE\tau_1(\eta)}  \int_\time^\infty \big[ \PP(\tau_j(\eta) >u) - \PP(\tau_{j-1}(\eta) >u) \big] du .
\end{equation}
This in particular generalizes the identities for ergodic maps in \cite{CK}  to (not necessarily ergodic) semi-flows. An analogous relation holds for the number of hits in a given time interval \cite{Haydn07}.

\section{Convergence}\label{sec:shrinking}

We now consider the entry and return time statistics for a sequence of probability spaces $(Y_n,\curA_n,\mu_n)$, a sequence of measurable maps $T_n:Y_n\to Y_n$ preserving $\mu_n$, target sets $D_n\in\curA_n$ and roof functions $r_n\in\L^1(Y_n,\mu_n)$ with $\inf r_n>0$ for every $n$. We rescale the entry times by the intensity, i.e., define the sequence of entry times $0< t_{1,n}<t_{2,n}<\ldots$ as the elements of the set $H_n(x,D_n)=\overline r_n^{-1} \mu_n(D_n) H(x,D_n)$, and put
\begin{equation}\label{eq:Pn}
P_n: X \to \scrN(\RR_{>0}), \qquad x\mapsto  \sum_{t\in H_n(x,D_n)} \delta_t .
\end{equation}
One example to keep in mind is when $(Y_n,\curA_n,\mu_n)=(Y,\curA,\mu)$  and $T_n=T$ are fixed, and one considers a sequence of targets $D_n\in \curA$ with $\mu(D_n)\to 0$. A further possibility is to fix $D$ and consider a sequence of invariant measures $\mu_n\in\curA$. The latter setting is for instance relevant for the gap statistics of Farey fractions via the horocycle flow: Take $\nu_n$ to be the invariant measure supported on a closed horocycle of length $\ell_n\to\infty$, and $\mu_n$ the measure supported on periodic points of the return map $T$ for a certain section (the Boca-Cobeli-Zaharescu map) \cite{Athreya14}. The equidistribution of long closed horocycles implies the weak convergence $\nu_n\toweak\nu$ and $\mu_n\toweak\mu$, which in turn implies $\xi_n\todist\xi$ in this setting; see \cite{Athreya16} for more examples of this type and \cite{Marklof13} for the higher-dimensional analogue. 

For $x$ random with respect to $\nu_n$ and $\nu_n$ defined by \eqref{Santalo} with respect to $\mu_n$, denote by $\xi_n$ be the stationary extension of $P_n(x)$. In the present scaling $I_{\xi_n}=1$ and $\EE\tau_1(\eta_n)=\PP(\xi_n\neq 0)$, where $\eta_n$ is distributed according to $Q_{\xi_n}$.The following well known fact deals with slightly more general point processes with not necessarily constant, but still uniformly bounded, intensity.

\begin{prop}\label{prop:relcomp}
Let $(\xi_n)$ be a sequence of stationary point processes with $\sup_n I_{\xi_n}<\infty$. Then $(\xi_n)$ is relatively compact in distribution.
\end{prop}

\begin{proof}
By \cite[Lemma 16.15]{Kallenberg02}, we need to show that the sequence of random variables $(\xi_n B)$ is tight in $\RR_{\geq 0}$ for any relatively compact Borel set $B\subset\RR$. By Markov's inequality, for any $K>0$,
\begin{equation}
\PP(\xi_n B\geq K) \leq \frac1K \EE\xi_n B = \frac1K I_{\xi_n} \Leb B.
\end{equation}
Therefore
\begin{equation}
\limsup_{K\to\infty} \bigg(\sup_n \PP(\xi_n B\geq K) \bigg)=0
\end{equation}
and hence $(\xi_n B)$ is tight.
\end{proof}

Relative compactness means that any subsequence of $(\xi_n)$ contains a converging subsequence. Let $\xi$ be a limit point of a converging subsequence, i.e., $\xi_n\todist\xi$ along some subsequence. The limit will again be a stationary, but not necessarily simple, point process.

The following proposition shows that the convergence of entry times is equivalent to the convergence of return times. It is a direct consequence of Eq.~\eqref{HHH}, and generalizes the results of \cite{Haydn05} (see also \cite{Abadi11,Haydn14b}). Again we consider more general processes, not necessarily those with unit intensity.

\begin{prop}\label{cor:conve}
Let $(\xi_n)$ be a sequence of stationary point processes on $\RR$ with $0<I_{\xi_n}<\infty$, and $\eta_n$ a point process given by the Palm distribution of $\xi_n$ such that $\sup_n \EE\tau_1(\eta_n)<\infty$.
Let $\rho$ be a Borel probability measure on $\RR_{\geq 0}$ and $F(\time)=\rho(\RR_{>\time})$.
Then the following are equivalent.
\begin{enumerate}[(i)]
\item $\EE\tau_1(\eta_n)\; \PP(\tau_1(\xi_n)\leq \time \mid \xi_n\neq 0) \to \int_0^\time F(u) du$ for every $\time\geq 0$;
\item $\EE g(\tau_1(\eta_n)) \to \int_0^\infty g(u) \rho(du)$
for every bounded continuous $g:\RR_{\geq 0} \to \RR$.
\end{enumerate}
\end{prop}

A simple example illustrating this statement is given in Appendix \ref{app:singular}.

\begin{proof}
By \eqref{HHH},
\begin{align}
\EE\tau_1(\eta_n)\; \PP(\tau_1(\xi_n)\leq \time \mid\xi\neq 0) 
& = \int_0^\time \PP(\tau_1(\eta_n) >u) du \\ 
& = \EE \int_0^\time \ind{\tau_1(\eta_n) >u} du \\
& = \EE g(\tau_1(\eta_n)) , 
\end{align}
with $g(u)=\min(u,\time)$, which is bounded continuous on $\RR_{\geq 0}$ with $g(0)=0$. So (ii) implies
\begin{align}\label{but}
\EE\tau_1(\eta_n)\; \PP(\tau_1(\xi_n)\leq\time \mid\xi\neq 0)  & \to \int_0^\infty g(u) \rho(du) \\
& = \int_0^\time u\rho(du)+\time F(\time) \\
& = \int_0^\time F(u) du ,
\end{align}
where the last equality follows from integration by parts. This proves (i).

Now assume (i) holds. Since, by Markov's inequality and \eqref{eq:Kac}, for any $K>0$
\begin{equation}
\PP(\tau_1(\eta_n) \geq K) \leq \frac1K \EE \tau_1(\eta_n) ,  
\end{equation}
where $ \EE \tau_1(\eta_n)$ has by assumption a uniform upper bound, the sequence $(\tau_1(\eta_n) )$ is tight and hence relatively compact.
This means that every subsequence contains a convergent subsequence, whose weak limit is a Borel probability measure $\rho_0$, i.e.,  for any bounded continuous $g:\RR_{\geq 0}\to \RR$,
\begin{equation}
\EE g(\tau_1(\eta_n)) \to \int_0^\infty g(u) \rho_0(du) .
\end{equation}
But the already established implication (ii) $\Rightarrow$ (i) then yields for $g(u)=\min(u,\time)$ and all $\time\geq 0$
\begin{equation}\label{but2}
 \EE \tau_1(\eta_n) \; \PP(\tau_1(\xi_n)\leq\time)  \to \int_0^\infty g(u) \rho_0(du) =  \int_0^\time F_0(u) du
\end{equation}
with $F_0(\time)=\rho_0(\RR_{>\time})$. This implies by assumption (i)
\begin{equation}
\int_0^\time F_0(u) du = \int_0^\time F(u) du .
\end{equation}
This yields $F_0=F$ (since $F_0,F$ are right-continuous), and hence $\rho_0=\rho$ (since $\rho_0,\rho$ are Borel probability measures on $\RR_{\geq 0}$). Thus the weak limit $\rho_0=\rho$ is unique, which in turn implies that every subsequence converges. This proves (ii).
\end{proof}

The following theorem complements Proposition \ref{cor:conve} by providing a criterion for the equivalence  of convergence for general stationary processes and their Palm distributions. This is a special case of Theorem 6.1 in \cite{Kallenberg73} for general point processes (see the remark after Theorem 6.3). 

\begin{prop}\label{prop:conve}
Let $(\xi_n)$ be a sequence of stationary point processes on $\RR$ with $0<I_{\xi_n}<\infty$, and $\eta_n$ a point process given by the Palm distribution of $\xi_n$. Then any two of the following statements imply the third: 
\begin{enumerate}[(i)]
\item $I_{\xi_n} \to I_\xi$; 
\item $\xi_n\todist \xi$; 
\item $\eta_n\todist \eta$, where $\eta$ has distribution $Q_\xi$. 
\end{enumerate}
\end{prop}

Applied to the processes generated by the entry and return times via \eqref{eq:Pn}, Proposition \ref{prop:conve} yields a generalization of Theorem 3.1 in \cite{Zweimueller16} to (not necessarily ergodic) continuous-time dynamical systems.

We conclude this section with an interpretation of the limit relations in Proposition \ref{cor:conve} in terms of the limit point processes $\xi$, $\eta$ in Proposition \ref{prop:conve}. Let $\xi^*$, $\eta^*$ be the simple point processes obtained by removing from $\xi$, $\eta$ all multiplicities. $\xi^*$ is again stationary and the intensities satisfy $I_\xi=m_\xi I_{\xi^*}$, where $m_\xi\geq 1$ is called the mean multiplicity of $\xi$. Furthermore, $\eta^*$ is given by the Palm distribution $Q_{\xi^*}$. Note that, by definition, $\tau_0(\zeta)=\tau_0(\zeta^*)$ and $\tau_1(\zeta)=\tau_1(\zeta^*)$ for any $\zeta\in\scrN(\RR)$. Thus \eqref{HHH} applied to $\xi^*$ yields
\begin{equation}
\PP(\tau_1(\xi)>\time \mid \xi\neq 0) = \frac{1}{m_\xi \EE\tau_1(\eta)} \int_\time^\infty \PP(\tau_1(\eta) >u) du . \label{HHH2}
\end{equation}
Comparing this formula with the limit distributions in Proposition \ref{cor:conve}, we have
\begin{equation}
 \EE \tau_1(\eta) \; \PP(\tau_1(\xi)\leq \time \mid \xi\neq 0) =  \int_0^\time F(u) du
\end{equation} 
and
\begin{equation}
\PP(\tau_1(\eta)>\time) = \widetilde\rho(\RR_{>\time})= m_\xi\, \rho(\RR_{>\time})=m_\xi\, F(\time)  ,
\end{equation}
with the Borel probability measure $\widetilde\rho$ defined by
\begin{equation}
\rho=\frac{m_\xi-1}{m_\xi} \delta_0 + \frac1{m_\xi} \widetilde\rho .
\end{equation}
The point here is that $\rho$ in Proposition \ref{cor:conve} registers return times of the sequence of point processes that accumulate at zero, whereas $\widetilde\rho$, as the return time distribution of the limit point process, ignores the corresponding zero return times.

\begin{appendix}

\section{Case study of an elementary non-ergodic suspension}\label{app:simple}

Fix $q_1\in[0,1]$, $q_0=1-q_1$, $\ell_0,\ell_1\in\RR_{>0}$.
Let $Y=\{0,1\}$ be the set of two elements, with the probability measure $\mu$ defined by $\mu\{i\}=q_i$, and $T$ be the identity map on $Y$. The suspension $\varphi^t$ on $X$ is defined with roof function $r(i)=\ell_i$. Thus $X$ is the union of two disjoint circles of length and $\ell_0$ and $\ell_1$, and $\varphi^t$ is the translation by $t$ on each circle. Each circle yields one ergodic component. The $\varphi^t$-invariant measure $\nu$ on $X$ as in \eqref{Santalo} has normalization $\overline r= q_0\ell_0+q_1\ell_1$. 

If the target set is for example $D=\{1\}$, then the entry times are $t_j=\ell_1 j-u_1$ with probability 
\begin{equation}\label{p1}
p_1=\nu(\{1\}\times[0,\ell_1)) =\frac{q_1\ell_1}{q_0\ell_0+q_1\ell_1},
\end{equation}
and $u_1$ distributed according to the uniform probability measure on $[0,\ell_1)$, and undefined with probability $p_0=1-p_1$. The point process $\xi$ of the entry times, extended to a stationary process on $\RR$, is thus
\begin{equation}
\xi=\alpha \sum_{m\in\ZZ} \delta_{\ell_1 m+s_1} ,
\end{equation}
where $\alpha=1$ with probability $p_1$, $\alpha=0$ with probability $p_0=1-p_1$, and $s_1\in[0,\ell_1]$ is a uniformly distributed random variable on $[0,\ell_1]$. The point process $\xi$  has intensity 
\begin{equation}
I_\xi= \EE\xi[0,1]=\frac{p_1}{\ell_1}=\frac{q_1}{q_0\ell_0+q_1\ell_1} = \overline r^{-1} \mu(\{1\}),
\end{equation}
as predicted by Proposition \ref{prop:stat}. 

The Palm distribution \eqref{eq:Q} of $\xi$ is given by 
\begin{align}\label{eq:QQQ300}
Q_\xi f & = \frac{1}{I_\xi \Leb B} \frac{p_1}{\ell_1}\int_0^{\ell_1} \int_B  \sum_{m\in\ZZ} f\bigg(\sum_{n\in\ZZ} \delta_{\ell_1 n +s-u}\bigg) \delta_{m\ell_1+s}(du) ds \\
& = \frac{1}{\Leb B} \int_\RR \int_B  f\bigg(\sum_{n\in\ZZ} \delta_{\ell_1 n+s -u}\bigg) \delta_{s}(du) ds \\
& =  f\bigg(\sum_{n\in\ZZ} \delta_{\ell_1 n}\bigg) .
\end{align}
Thus $Q_\xi$ is a point mass, and a random measure $\eta$ with distribution $Q_\xi$ satisfies
\begin{equation}
\eta= \sum_{n\in\ZZ} \delta_{\ell_1 n} \quad \text{a.s.}
\end{equation}
The first hitting time of $\eta$ is $\tau_1(\eta) = \ell_1$ a.s., and so $\EE\tau_1(\eta) =\ell_1 = p_1 I_\xi^{-1}$, which confirms \eqref{eq:Kac}.

If the target set is $D=Y$, then the entry times are $t_j=\ell_1 j-u_1$ with probability $p_1$ as in \eqref{p1}, and $t_j=\ell_0 j-u_0$ with probability $p_0=1-p_1$. The random variables $u_0,u_1$ are independent and uniformly distributed on $[0,\ell_0)$ and $[0,\ell_1)$, respectively. The point process $\xi$ of the entry times, extended to a stationary process on $\RR$, is now
\begin{equation}\label{thisxi}
\xi=\alpha \sum_{m\in\ZZ} \delta_{\ell_1 m+s_1} + (1-\alpha) \sum_{m\in\ZZ} \delta_{\ell_0 m+s_0} 
\end{equation}
where $\alpha=1$ with probability $p_1$, $\alpha=0$ with probability $p_0=1-p_1$, and $s_0\in[0,\ell_0]$, $s_1\in[0,\ell_1]$ are uniformly distributed, independent random variables.  

We have $\PP(\xi\neq 0)=1$, and $\xi$ is a stationary, simple point process with intensity 
\begin{equation}
I_\xi = \EE\xi[0,1]=\frac{p_0}{\ell_0}+\frac{p_1}{\ell_1}=\frac{1}{q_0\ell_0+q_1\ell_1} = \overline r^{-1} >0. 
\end{equation}
By the same calculation as in \eqref{eq:QQQ300}, the Palm distribution \eqref{eq:Q} of $\xi$ is now
\begin{equation}\label{eq:QQQ3}
Q_\xi f  = \frac{1}{I_\xi} \sum_{i=0,1}  \frac{p_i}{\ell_i}  f\bigg(\sum_{n\in\ZZ} \delta_{\ell_i n}\bigg) .
\end{equation}
A random measure $\eta$ with distribution $Q_\xi$ can be therefore realized by the simple point process 
\begin{equation}
\eta = \sum_{n\in\ZZ} \delta_{\ell n} 
\end{equation}
with
\begin{equation}
\EE f(\eta) = \frac{1}{I_\xi} \EE \bigg[\frac{1}{\ell}  f\bigg(\sum_{n\in\ZZ} \delta_{\ell n}\bigg) \bigg],
\end{equation}
where the random variable $\ell$ takes the value $\ell_i$ with probability $p_i$.
The first hitting time of $\eta$ is $\tau_1(\eta) \eqdist \ell$, and so $\EE\tau_1(\eta) = I_\xi^{-1}$,
which is precisely \eqref{eq:Kac}.

\section{An illustrative example for Propositions \ref{cor:conve} and \ref{prop:conve}}\label{app:singular}

Consider the sequence of stationary, simple point processes
\begin{equation}
\xi_n = \sum_{m\in\ZZ} \sum_{k=-n}^{n} \delta_{(2n+1) (m+s) + a_n k} \, ,
\end{equation}
where $(a_n)$ is any sequence with $0<a_n<1$ and $\lim a_n= 0$, and the random variable $s$ is uniformly distributed in $\RR/\ZZ$. Note that $I_{\xi_n}=1$. 
The Palm distribution \eqref{eq:Q} of $\xi_n$ is given by 
\begin{equation}\label{eq:QQQ30011}
Q_{\xi_n} f = \frac{1}{2n+1} \sum_{l=-n}^n f\bigg(  \sum_{m\in\ZZ} \sum_{k=-n}^{n} \delta_{(2n+1) m + a_n (k-l)} \bigg) ,
\end{equation}
and hence we may choose
\begin{equation}
\eta_n =  \sum_{m\in\ZZ} \sum_{k=-n}^{n} \delta_{(2n+1) m + a_n (k-l)} ,
\end{equation}
where $l$ is a uniformly distributed random integer in $[-n,n]$. Hence 
\begin{equation}
\begin{cases}
\PP(\tau_1(\eta_n)=a_n) & = \frac{2n}{2n+1} \\
\PP(\tau_1(\eta_n)=2n+1 - 2na_n) & = \frac{1}{2n+1} \\
\PP(\tau_1(\eta_n)\notin\{a_n,2n+1 - 2na_n\}) & = 0 . 
\end{cases}
\end{equation}
This shows that 
\begin{equation}
\PP(\tau_1(\xi_n)\leq \time)\to 0, \quad \EE\tau_1(\eta_n)=1, \quad
\PP(\tau_1(\eta_n)>\time)\to 0,
\end{equation}
for each fixed $\time> 0$, and thus $F=0$ and $\rho=\delta_0$, which is consistent with Proposition \ref{cor:conve}.

This example does not, however, satisfy the hypotheses of Proposition \ref{prop:conve}. We have $\xi_n\toas 0$,  and hence $I_\xi=0$. On the other hand $I_{\xi_n}=1$, so hypothesis (i) in Proposition \ref{prop:conve} fails, whilst (ii) holds. Let us show that (iii) fails. Since $\eta_n[-1,1]\geq \min(\lfloor a_n^{-1}\rfloor, n)$ for every realization of $\eta_n$, the sequence $(\eta_n)$ accumulates infinite mass in $[-1,1]$, and hence does not converge in distribution to a random measure in $\scrM(\RR)$.

\end{appendix}

\end{document}